\documentclass[12pt]{article}

\usepackage{amsfonts}
\usepackage{graphicx}
\usepackage{amsmath}
\usepackage{amssymb}
\usepackage{amsthm}
\usepackage{url}
\usepackage{cite}
\usepackage{color}

\theoremstyle{plain}
\newtheorem{theorem}{Theorem}

\newtheorem{lemma}[theorem]{Lemma}

\theoremstyle{remark}
\newtheorem{remark}[theorem]{Remark}
\newtheorem{example}[theorem]{Example}

\numberwithin{equation}{section}
\numberwithin{theorem}{section}

\title{On geometric-type approximations with applications}
\author{Fraser Daly\footnote{Corresponding author. Heriot--Watt University, School of Mathematical and Computer Sciences, Edinburgh EH14 4AS, UK. Email \texttt{f.daly@hw.ac.uk}} \;and Claude Lef\`evre\footnote{Universit\'e Libre de Bruxelles, D\'epartement de Math\'ematique, Campus de la Plaine C. P. 210, B-1050 Bruxelles, Belgium. Email \texttt{claude.lefevre@ulb.be}}}
\date{\today}

\begin{document}

\maketitle

\begin{abstract}
We explore two aspects of geometric approximation via a coupling approach to Stein's method. Firstly, we refine precision and increase scope for applications by convoluting the approximating geometric distribution with a simple translation selected based on the problem at hand. Secondly, we give applications to several stochastic processes, including the approximation of Poisson processes with random time horizons and Markov chain hitting times. Particular attention is given to geometric approximation of random sums, for which explicit bounds are established. These are applied to give simple approximations, including error bounds, for the infinite-horizon ruin probability in the compound binomial risk process.
\end{abstract}

\vspace{12pt}


\vspace{12pt}


\section{Introduction and motivation}

Approximation by the geometric distribution and its generalisations have continued to receive particular attention in recent years, especially in conjunction with Stein's method, a powerful technique for probability approximation that has been used successfully for a variety of theoretical and applied topics. Stein's method was first introduced in \cite{s72}; see also \cite{r11} for a more recent introduction to the area. Stein's method for geometric approximation was developed by Barbour and Gr\"ubel \cite{b95} and Pek\"oz \cite{p96}. We refer the reader to \cite{p13} and references therein for more recent developments. 

Our goal in this paper is twofold. Our first aim is to consider how we can improve upon geometric approximation by incorporating a convolution with another random variable. Though a geometric approximation can yield effective error bounds in some applications, extra precision can be gained, and the range of useful applications increased, by using an approximating distribution that incorporates some additional parameters. Examples of this approach include negative binomial approximation (see \cite{r13,y23} and references therein) and approximation by a geometric sum \cite{d10,d16,d19}. Such approaches naturally lead to additional technical complexity in the evaluation of the approximating distribution or the corresponding error bounds. Our approach of convoluting the approximating geometric distribution with an independent translation is an alternative to these ideas. The mass function of this convolution must be kept simple to use, otherwise the benefits are lost. We are motivated here by the work of R\"ollin \cite{r07} on translated Poisson approximation, where better error bounds are obtained than using a traditional Poisson approximation. Our translation will allow us to introduce additional parameters in the approximating distribution which can then be chosen to gain precision, at the cost of some additional complexity in the approximation. Unlike the usual approach to translated Poisson approximation where the translation is by a constant, we will here allow translation by an independent random variable.

Our second goal is to consider geometric approximation of random sums (i.e., of sums of a random number of random variables) and related applications to risk processes. The case in which the number of terms in the sum is itself geometrically distributed has previously been treated by \cite{p13}; here we consider the more general case. Our results are applied to give a simple approximation, with explicit error bounds, of the infinite-horizon ruin probability in the compound binomial risk process \cite{g88}. As further applications of our general results, we will also give explicit error bounds in the approximation of the number of points in a Poisson process over a random time horizon, and of Markov chain hitting times. 

In these applications, and in particular in the approximation of random sums, there are a number of other approximating distributions available, including compound Poisson approximations (see, for example, \cite{v96}) which offer a flexible approximation framework. A geometric distribution is a simple special case of a compound Poisson distribution. This simplicity leads to more simply stated and easily calculated error bounds compared to a more general compound Poisson setting. As we shall see in the examples and applications of Sections \ref{sec:trans} and \ref{sec:sums} below, there are distributions of interest which are close to geometric, so this approximation is an appropriate and useful one to consider here.  
  
Recall that $Y\sim\textrm{Geom}(p)$ has a geometric distribution if $\mathbb{P}(Y=k)=p(1-p)^k$ for $k\in\mathbb{Z}^+=\{0,1,\ldots\}$. In evaluating the approximation error in the work that follows, we will employ the total variation distance between two non-negative, integer-valued random variables $W$ and $Y$, which can be defined in any of the following equivalent ways:
\begin{multline}\label{dtvdef}
d_{TV}(\mathcal{L}(W),\mathcal{L}(Y)) = \sup_{A\subseteq\mathbb{Z}^+}|\mathbb{P}(W\in A) - \mathbb{P}(Y\in A)|\\
 = \frac{1}{2}\sum_{j\in\mathbb{Z}^+}|\mathbb{P}(W=j)-\mathbb{P}(Y=j)| = \inf \mathbb{P}(W\not=Y)\,,
\end{multline}
where $\mathcal{L}(W)$ denotes the law of $W$, and the infimum is taken over all couplings of $W$ and $Y$ with the required marginals. 

The remainder of this paper is organised as follows. In Section \ref{sec:trans} we introduce Stein's method for geometric approximation, and use these techniques to derive results in the setting where we approximate by the convolution of a geometric distribution and an independent translation. This is illustrated by applications to Poisson processes and Markov chain hitting times. We then consider geometric approximation of random sums in Section \ref{sec:sums}, and apply these results to derive simple approximations for the ruin probability in the compound binomial risk process. 


\section{Translated geometric approximation}\label{sec:trans}

Following Stein's method for geometric approximation \cite{b95,p96}, for each $A\subseteq\mathbb{Z}^+$ we let $g_A:\mathbb{Z}^+\to\mathbb{R}$ be such that $g_A(k)=0$ for 
$k\leq0$ and
\begin{equation}\label{steineq}
I(k\in A)-\mathbb{P}(Y\in A) = (1-p)g_A(k+1) - g_A(k)\,,
\end{equation}
for some $p\in(0,1]$, where $Y\sim\mbox{Geom}(p)$. Recall that Stein's method proceeds by replacing $k$ in \eqref{steineq} by the random variable $W$ to be approximated, taking expectations, and bounding the expression thus obtained from the right-hand side of \eqref{steineq}. Essential ingredients in the final part of this procedure include bounds on the behaviour of the functions $g_A$.  For example, 
\begin{equation}
\label{magic2}\sup_{A\subseteq\mathbb{Z}^+}\sup_{j,k\in\mathbb{Z}}|g_A(j)-g_A(j+k)|\leq\min\left\{|k|,\frac{1}{p}\right\}\,,
\end{equation}
by Lemma 4.2 of \cite{d16} and Lemma 1 of \cite{p96}.

In the work that follows, we let $W$ denote the non-negative, integer-valued random variable we wish to approximate.  We consider the approximation of $W$ by $Y+T$, 
where $Y\sim\mbox{Geom}(p)$ for some $p$, and $T$ is our random translation independent of $W$ and $Y$. Our results will employ a coupling 
construction based on that used by \cite{p96} and \cite{d10}.  To that end, we define the random variable $V$ to be such that
\begin{equation} \label{vdef}
\mathcal{L}(V+T+1) = \mathcal{L}(W|W>T)\,,
\end{equation}
emphasising that $V$ will depend on both $T$ and $W$.

With this setup, we may obtain the following
\begin{theorem} \label{thm1}
Let $W$ and $T$ be independent integer-valued random variables with $W\geq0$.  Suppose $V$ is as defined by \eqref{vdef} and let $Y\sim\mbox{Geom}(p)$ be independent of $T$, where $p=\mathbb{P}(W\leq T)$.  Then
\begin{align}
\label{thm1a} d_{TV}(\mathcal{L}(W),\mathcal{L}(Y+T)) &\leq \mathbb{P}(W<T) + (1-p)\mathbb{E}|W-T-V|\,,\\
\label{thm1b} d_{TV}(\mathcal{L}(W),\mathcal{L}(Y+T)) &\leq \mathbb{P}(W<T) + \frac{1-p}{p}d_{TV}(\mathcal{L}(W-T),\mathcal{L}(V))\,.
\end{align}
\end{theorem}
\begin{proof}
Since $T$ is independent of $W$ and $Y$, 
\[
d_{TV}(\mathcal{L}(W),\mathcal{L}(Y+T))=d_{TV}(\mathcal{L}(W-T),\mathcal{L}(Y))\,.
\]
Motivated by the argument of (3.11) of \cite{r07}, we now write
\begin{align*}
d_{TV}(\mathcal{L}(W-T),\mathcal{L}(Y))&=\sup_{A\subseteq\mathbb{Z}}\left|\mathbb{P}(W-T\in A)-\mathbb{P}(Y\in A)\right|\\
&\leq \mathbb{P}(W<T) + \sup_{A\subseteq\mathbb{Z}^+}\left|\mathbb{E}[(1-p)g_A(W-T+1)-g_A(W-T)]\right|\,,
\end{align*}
where we recall that $g_A(k)=0$ for $k\leq0$ and use \eqref{steineq} with $k=W-T$ for the final expression. Now, conditioning on the value of $W$,
\[
\mathbb{E}[g_A(W-T)] = (1-p)\mathbb{E}[g_A(W-T)|W>T] +p\mathbb{E}[g_A(W-T)|W\leq T] = (1-p)\mathbb{E}[g_A(V+1)]\,,
\]
using the definitions of $g_A$ and $V$. Our bound \eqref{thm1a} then follows from \eqref{magic2} when writing
\[
|\mathbb{E}[g_A(W-T+1)-g_A(V+1)]|\leq\mathbb{E}|W-T-V|\,.
\]
Similarly, for \eqref{thm1b} we let $I(\cdot)$ denote an indicator function and write
\begin{align*}
|\mathbb{E}[g_A(W-T+1)-g_A(V+1)]|&=|\mathbb{E}[\left\{g_A(W-T+1)-g_A(V+1)\right\}I(W-T\not=V)]|\\
&\leq\frac{1}{p}d_{TV}(\mathcal{L}(W-T),\mathcal{L}(V))\,,
\end{align*}
using the definition \eqref{dtvdef} of total variation distance and the bound \eqref{magic2}.
\end{proof}

\begin{remark}
Depending on the situation, it is sometimes more natural to define a geometric random variable on $\{1,2\ldots\}$ rather than $\{0,1,\ldots\}$. In our framework it is straightforward to switch between the two definitions by simply altering the definition of $T$.
\end{remark}

In the remainder of this section we consider two applications of Theorem \ref{thm1}, to approximation of the number of points observed in a Poisson process over a random time horizon and of Markov chain hitting times. For later use, we recall that a random variable $U$ is said to be stochastically larger than a random variable $V$ if $\mathbb{P}(U>t)\geq\mathbb{P}(V>t)$ for all $t$, and that this is equivalent to the existence of a probability space on which we may construct the pair $(U^\prime,V^\prime)$ such that $\mathcal{L}(U^\prime)=\mathcal{L}(U)$, $\mathcal{L}(V^\prime)=\mathcal{L}(V)$, and $U^\prime\geq V^\prime$ almost surely; see Theorem 1.A.1 of \cite{s07}. We denote this by $U\geq_{st}V$.

\subsection{Application: Poisson processes with random time horizon}

Let $\{N(t):t\geq0\}$ be a homogeneous Poisson process of rate $\lambda$, and $\tau$ be a non-negative random variable independent of this process. We consider the approximation of $W=N(\tau)$ by $Y+T$, where $Y$ has a geometric distribution and $T$ is a non-negative, integer-valued random variable independent of $W$ and $Y$. In this we are motivated by \cite{d16}, who previously considered geometric-type approximations in this setting, and by applications considered by \cite{a02} and others in which Poisson processes with random time horizons play a role. We expect a geometric-type approximation to be relevant in cases of practical importance, since if $\tau$ has an exponential distribution then $N(\tau)$ is precisely geometrically distributed.

Let $\xi_0=0$ and, for $i\geq1$, let $\xi_i$ be the time of the $i$th event in our Poisson process, so that $N(t)=\max\{i:\xi_i\leq t\}$. We define $V=N(\sigma_T)$, where 
\[
\mathcal{L}(\sigma_T)=\mathcal{L}(\tau-\xi_{T+1}|\tau>\xi_{T+1})
\]
and $\sigma_T$ depends on the random variables $\{N(t):t\geq0\}$, through $\xi_{T+1}$, and $\tau$. This $V$ satisfies \eqref{vdef}. To see this, note that conditioning on $N(\tau)>T$ is equivalent to conditioning on the $(T+1)$th event occurring before time $\tau$, i.e., that $\xi_{T+1}\leq\tau$. Conditioning on this, the number of events we have seen up to time $\tau$ is then $T+1+V$. 

We thus obtain from \eqref{thm1a} that
\begin{equation}\label{renewal}
d_{TV}(\mathcal{L}(N(\tau)),\mathcal{L}(Y+T))\leq\mathbb{P}(\tau<\xi_T)+(1-p)\mathbb{E}|N(\tau)-T-N(\sigma_T)|\,,
\end{equation}
where $Y\sim\mbox{Geom}(p)$ and $p=\mathbb{P}(\tau<\xi_{T+1})$.

Consider, for simplicity, the case where $T=0$ almost surely, and write $\sigma=\sigma_0$. Note that 
\begin{equation}\label{sigma}
\mathbb{P}(\sigma>u)=\frac{\mathbb{P}(\tau>u+\xi_1)}{\mathbb{P}(\tau>\xi_1)}\,,
\end{equation}
and
\begin{align*}
\mathbb{E}[\sigma]&=\int_0^\infty\mathbb{P}(\sigma>u)\,\text{d}u=\frac{1}{1-p}\mathbb{E}\int_0^\infty\mathbb{P}(\tau>\xi_1+u|\xi_1)\,\text{d}u\\
&=\frac{\mathbb{E[\tau]}}{1-p}-\frac{1}{1-p}\int_0^\infty\int_0^{x}\mathbb{P}(\tau>u)\lambda e^{-\lambda x}\,\text{d}u\,\text{d}x\\
&=\frac{\mathbb{E[\tau]}}{1-p}-\frac{1}{\lambda}\,,
\end{align*}
where $p=\mathbb{P}(\tau<\xi_1)=\mathbb{E}[e^{-\lambda\tau}]$.

In the case where $\sigma\geq_{st}\tau$ the bound \eqref{renewal} simplifies considerably, since we then have
\begin{equation}\label{eq:nwu_mean}
\mathbb{E}|N(\tau)-N(\sigma)|=\mathbb{E}N(\sigma-\tau)=\lambda\mathbb{E}[\sigma-\tau]
=\frac{\lambda p\mathbb{E}[\tau]-(1-p)}{1-p}=\frac{p(1+\lambda\mathbb{E}[\tau])-1}{1-p}\,.
\end{equation}
Similarly, if $\tau\geq_{st}\sigma$ we have 
\begin{equation}\label{eq:nbu_mean}
\mathbb{E}|N(\tau)-N(\sigma)|=\frac{1-p(1+\lambda\mathbb{E}[\tau])}{1-p}\,.
\end{equation}

Observe that if $\tau$ satisfies the new worse than used (NWU) property that $\mathbb{P}(\tau>u)\mathbb{P}(\tau>v)\leq\mathbb{P}(\tau>u+v)$ for all $u,v\geq0$, the independence of $\tau$ and $\xi_1$ combined with \eqref{sigma} implies that $\sigma\geq_{st}\tau$. Similarly, if $\tau$ satisfies the new better than used (NBU) property that $\mathbb{P}(\tau>u)\mathbb{P}(\tau>v)\geq\mathbb{P}(\tau>u+v)$ for all $u,v\geq0$, we have that $\tau\geq_{st}\sigma$. By combining the bound \eqref{renewal} with the choice $p=\mathbb{E}[e^{-\lambda\tau}]$ and the expectations calculated in \eqref{eq:nwu_mean} and \eqref{eq:nbu_mean} we thus have the following
\begin{theorem}\label{thm_po}
Let $\{N(t):t\geq0\}$ be a homogeneous Poisson process of rate $\lambda$. Let $\tau$ be a non-negative random variable independent of this process.
\begin{enumerate}
\item[(a)] If $\tau$ is NBU,
\[
d_{TV}(\mathcal{L}(N(\tau)),\text{Geom}(\mathbb{E}[e^{-\lambda\tau}]))\leq1-\mathbb{E}[e^{-\lambda\tau}](1+\lambda\mathbb{E}[\tau])\,.
\]
\item[(b)] If $\tau$ is NWU,
\[
d_{TV}(\mathcal{L}(N(\tau)),\text{Geom}(\mathbb{E}[e^{-\lambda\tau}]))\leq\mathbb{E}[e^{-\lambda\tau}](1+\lambda\mathbb{E}[\tau])-1\,.
\]
\end{enumerate}
\end{theorem}
If $\tau$ is exponentially distributed (which is trivially both NBU and NWU) then $\tau$ and $\sigma=\sigma_0$ as above have the same distribution, and the upper bounds of Theorem \ref{thm_po} are each zero, reflecting the fact that $N(\tau)$ has a geometric distribution. 

The upper bound of Theorem \ref{thm_po}(a) was previously established in Section 2.2 of \cite{d16}, but only under the stronger condition that $\tau$ has increasing failure rate.

We conclude this section with two examples illustrating Theorem \ref{thm_po}.
 
\begin{example}
Let $\tau$ have a gamma distribution, and note that $N(\tau)$ is thus negative binomial. For simplicity of presentation we normalise $\tau$ to have unit mean, with density function $x\mapsto\frac{\beta^\beta}{\Gamma(\beta)}x^{\beta-1}e^{-\beta x}$ for $x>0$, where $\beta>0$ is a parameter and $\Gamma(\cdot)$ is the usual gamma function. We have that $\mathbb{E}[e^{-\lambda\tau}]=\left(1+\frac{\lambda}{\beta}\right)^{-\beta}$ in this case. Note that $\tau$ has decreasing failure rate when $\beta<1$ (and is thus NWU), constant failure rate when $\beta=1$ (the exponential case), and increasing failure rate when $\beta>1$ (and is thus NBU). Applying both parts of Theorem \ref{thm_po} we have
\begin{equation}\label{eq:NegBin}
d_{TV}\left(\mathcal{L}\left(N(\tau)\right),\mbox{Geom}\left(\left(1+\frac{\lambda}{\beta}\right)^{-\beta}\right)\right)\leq\left|1-(1+\lambda)\left(1+\frac{\lambda}{\beta}\right)^{-\beta}\right|\,.
\end{equation}
The upper bound \eqref{eq:NegBin} is illustrated in Figure \ref{fig:1} for several values of $\lambda$. As expected, this upper bound is zero when $\beta=1$, and is small when $\beta$ is close to $1$.
\begin{figure}
\begin{center}
\includegraphics[scale=0.5]{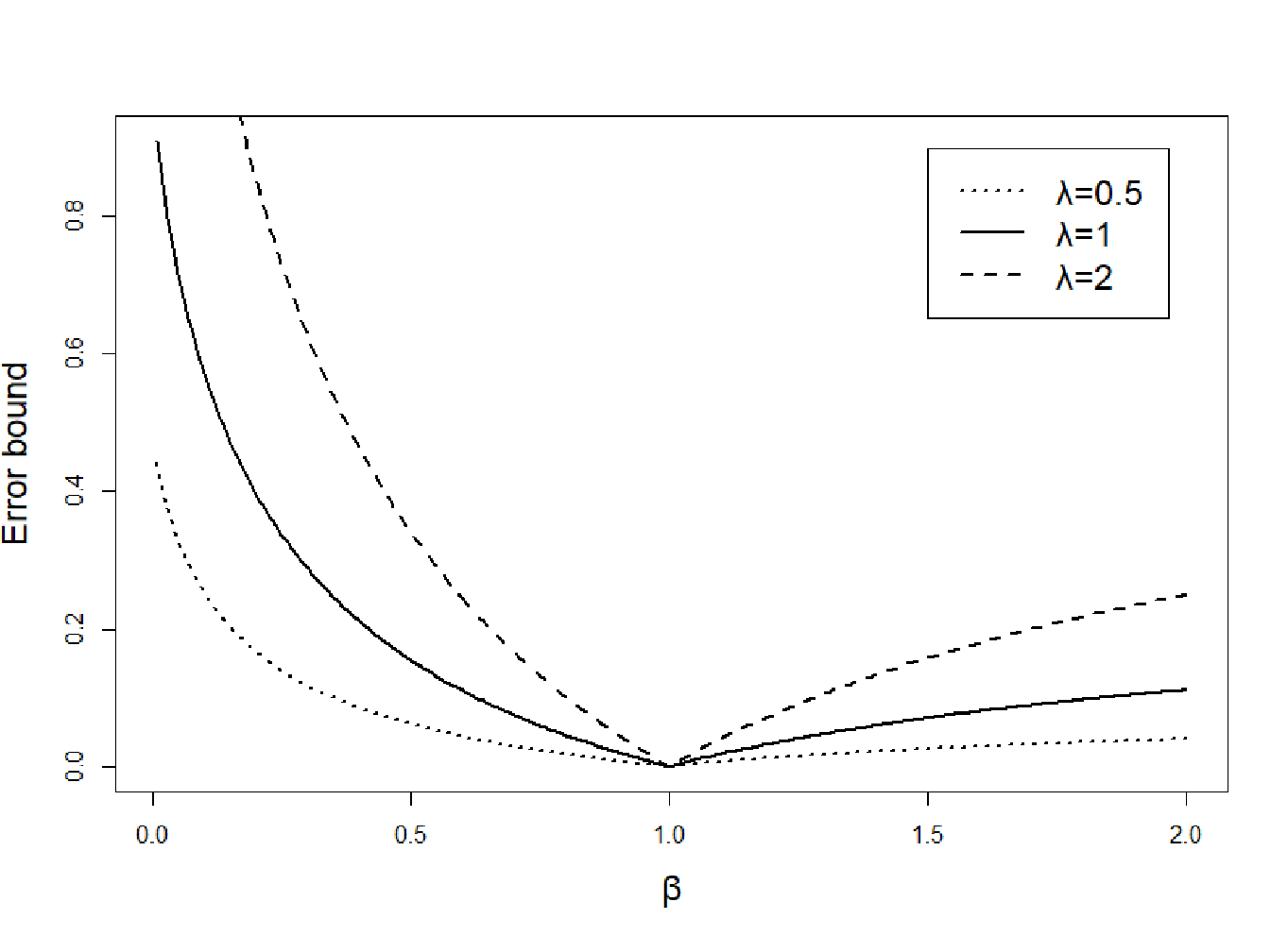}
\end{center}
\caption{The upper bound \eqref{eq:NegBin} as a function of $\beta$ for $\lambda=0.5,1,2$}\label{fig:1}
\end{figure}
\end{example}

\begin{example}
Let $\tau$ be uniformly distributed on the interval $[a,b]$ for some $0\leq a<b$, so that $\tau$ has increasing failure rate (and is thus NBU) and satisfies $\mathbb{E}[e^{-\lambda\tau}]=\frac{e^{-\lambda a}-e^{-\lambda b}}{\lambda(b-a)}$ for all $\lambda>0$. Hence, by Theorem \ref{thm_po}(a),
\[
d_{TV}\left(\mathcal{L}\left(N(\tau)\right),\mbox{Geom}\left(\frac{e^{-\lambda a}-e^{-\lambda b}}{\lambda(b-a)}\right)\right)\leq1-\frac{e^{-\lambda a}-e^{-\lambda b}}{b-a}\left(\lambda+\frac{a+b}{2}\right)\,.
\] 
If $\lambda a$ and $\lambda b$ are small, writing $e^{-\lambda a}\approx 1-\lambda a$ and $e^{-\lambda b}\approx 1-\lambda b$ we see that this upper bound is small when $a+b\approx\frac{2(1-\lambda^2)}{\lambda}$. 
\end{example}

\subsection{Application: Markov chain hitting times}

We now let $W$ be the hitting time of an ergodic Markov chain to a set of states $A$. That is, we let $(\xi_0,\xi_1,\ldots)$ be a Markov chain and $W=\min\{i\geq0:\xi_i\in A\}$ for some fixed subset $A$ of the state space. We will write $P_{ij}^{(n)}=\mathbb{P}(\xi_n=j|\xi_0=i)$ for the $n$-step transition probabilities and $\pi$ for the stationary distribution of our Markov chain, with $\pi(A)$ denoting the measure of a set $A$ according to $\pi$; abusing notation, we will write $\pi_i$ for $\pi(\{i\})$.   

For some distribution $F$ on this state space, we let $\xi_0\sim F$, and define $T=T(F)$ to be a stationary time for this initial distribution, independent of $W$. That is, $T$ is such that $\mathcal{L}(\xi_T|\xi_0\sim F)$ is $\pi$. In line with Theorem \ref{thm1}, we consider the approximation of $W$ by $Y+T$, where $Y\sim\mbox{Geom}(p)$ is a geometric random variable with parameter $p=\mathbb{P}(W\leq T)$ which is independent of $T$. We will use \eqref{thm1b} to give a bound. To that end, we construct the random variables $W-T$ and $V$ as follows: Consider the Markov chain $(X,\xi_0^\prime,\xi_1^\prime,\ldots)$, where the initial state $X$ is sampled according to $\pi$. Since $T$ is a stationary time, we may construct $W-T$ as the number of time steps required to hit $A$ from an initial stationary distribution. That is, we write $W-T=\min\{i\geq0:\xi_i^\prime\in A\}$. Now consider another Markov chain $(Y,\xi_0^{\prime\prime},\xi_1^{\prime\prime},\ldots)$ whose initial state $Y$ is sampled according to $\pi$ restricted to $A^\text{c}$, the complement of $A$. We may then write $V=\min\{i\geq0:\xi_i^{\prime\prime}\in A\}$ since, in conditioning on the event that $W>T$, we consider the number of steps from such an initial state until we hit $A$ in addition to the $T+1$ steps already taken to reach the state $\xi_0^{\prime\prime}$.

With these random variables, the maximal coupling argument of Theorem 2 of \cite{p96} gives
\[
d_{TV}(\mathcal{L}(W-T),\mathcal{L}(V))\leq\frac{1}{\pi(A^{\text{c}})}\sum_{i,j\in A}\pi_i\sum_{n=1}^\infty|P_{ij}^{(n)}-\pi_j|\,.
\]   
The following bound is then immediate from \eqref{thm1b}.
\begin{theorem}
Let $(\xi_0,\xi_1,\ldots)$ be the Markov chain above, $W=\min\{i\geq0:\xi_i\in A\}$, and $T$ be a stationary time as above. With $Y\sim\mbox{Geom}(p)$ for $p=\mathbb{P}(W\leq T)$,
\[
d_{TV}(\mathcal{L}(W),\mathcal{L}(Y+T))\leq\mathbb{P}(W<T)+\frac{1-p}{p\pi(A^{\text{c}})}\sum_{i,j\in A}\pi_i\sum_{n=1}^\infty|P_{ij}^{(n)}-\pi_j|\,.
\]
\end{theorem}
This result generalises Theorem 2 of \cite{p96}, where the hitting time from an initial stationary distribution only is considered, in which case our $T$ is equal to zero almost surely and $p=\pi(A)$.


\section{Geometric approximation of random sums and the compound binomial risk process} \label{sec:sums}

Throughout this section we let $W=\sum_{i=1}^NX_i$, where $X_1,X_2,\ldots$ are independent and identically distributed, positive integer-valued random variables with $\mathbb{E}[X_1]=\mu\geq1$. We let $N$ be a non-negative, integer-valued random variable independent of the $X_i$. For simplicity we restrict ourselves to the case where the $X_i$ are identically distributed, but this condition can easily be relaxed.

We will consider the approximation of $W$ by a geometric random variable $Y$ in Section \ref{subsec:sums}, and then apply our results in Section \ref{subsec:ruin} to derive simple and explicit bounds for the infinite-horizon ruin probability in the compound binomial risk process. 

In this section we focus on geometric approximation (i.e., $T=0$ almost surely in the setting of Theorem \ref{thm1}), since construction of the random variable $V$ defined in \eqref{vdef} seems to be natural here only when $T=0$. This case is also sufficient to give simple and explicit bounds on our ruin probability of interest, which is the principal application of this section.

\subsection{Geometric approximation of random sums}\label{subsec:sums}

With $W$ the random sum defined above, we now derive geometric approximation bounds for $W$. For later use, we define $N^{(0)}$ such that $\mathcal{L}(N^{(0)})=\mathcal{L}(N|N>0)$.

We consider the approximation of $W$ by $Y\sim\mbox{Geom}(p)$. In line with the discussion in Section \ref{sec:trans} in the case where $T=0$ almost surely, we choose the parameter $p$ to be
\begin{equation}\label{pdef}
p = \mathbb{P}(W=0) = \mathbb{P}(N=0)\,,
\end{equation}
and $V$ to be such that
\[
\mathcal{L}(V+1) = \mathcal{L}(W|N>0) = \mathcal{L}\left(\sum_{j=1}^{N^{(0)}}X_j\right)\,,
\]
so that we may choose $V=\sum_{j=1}^{N^{(0)}}X_j-1$. With this choice we have that
\[
\mathbb{E}|W-V|=\mathbb{E}\left|\sum_{i=1}^NX_i-\sum_{j=1}^{N^{(0)}}X_j+1\right|=\mathbb{E}\left|\sum_{i=N+1}^{N^{(0)}}X_i-1\right|\,,
\]
since $N^{(0)}\geq_{st}N$, and so we may construct these on the same probability space in such a way that $N^{(0)}\geq N$ almost surely.

Furthermore, if $N^{(0)}\geq_{st}N+1$ then we may construct our random variables such that $\sum_{i=N+1}^{N^{(0)}}X_i$ is almost surely greater than $1$, since each $X_i$ is at least $1$ almost surely. In this case we thus have that 
\[
\mathbb{E}|W-V|=\mu\mathbb{E}[N^{(0)}-N]-1=\frac{\mu\mathbb{P}(N=0)\mathbb{E}[N]}{\mathbb{P}(N>0)}-1=\frac{\mu p\mathbb{E}[N]}{1-p}-1\,.
\]
From the definition, we note that $N^{(0)}\geq_{st}N+1$ if
\begin{equation}\label{hazard}
\frac{\mathbb{P}(N=j)}{\mathbb{P}(N>j)}\leq\frac{p}{1-p}\,,
\end{equation}
for all $j\in\mathbb{Z}^+$. That is, $N^{(0)}\geq_{st}N+1$ if $N$ is larger than $Y$ in the hazard rate ordering; see Section 1.B of \cite{s07}.

Combining all these ingredients, we may apply \eqref{thm1a} with $T=0$ to give
\begin{theorem} \label{thm:sum}
Let $W=\sum_{i=1}^NX_i$ be as above and $p$ be given by \eqref{pdef}. Then
\[
d_{TV}(\mathcal{L}(W),\mbox{Geom}(p)) \leq (1-p)\mathbb{E}\left|\sum_{i=N+1}^{N^{(0)}}X_i-1\right|\,.
\]
Moreover, if \eqref{hazard} holds for all $j\in\mathbb{Z}^+$, then
\begin{equation}\label{geomsum}
d_{TV}(\mathcal{L}(W),\mbox{Geom}(p)) \leq p(\mu\mathbb{E}[N]+1)-1\,.
\end{equation} 
\end{theorem}
One case which has received particular attention in the literature is that in which $N\sim\mbox{Geom}(r)$ has a geometric distribution, so that $W$ is a geometric sum. Note that in this case $\mathcal{L}(N^{(0)})=\mathcal{L}(N+1)$, so that the bound \eqref{geomsum} applies. Geometric approximation for geometric sums has also been considered by \cite{p13}, who show in their Theorem 3.2 that
\begin{equation}\label{geomsum2}
d_{TV}(\mathcal{L}(W),\mbox{Geom}(p^\prime))\leq\frac{1}{2}\min\left\{1,r\left[1+\sqrt{\frac{-2}{u\log(1-r)}}\right]\right\}\left(\frac{\mathbb{E}[X_1^2]}{\mu}-1\right)\,,
\end{equation}
where $p^\prime=\frac{r}{r+\mu(1-r)}$ and $u$ is such that $d_{TV}(\mathcal{L}(X_1),\mathcal{L}(X_1+1))\leq1-u$. Note that this parameter $p^\prime$ is chosen such that the mean of the approximating geometric distribution matches that of $W$; on the other hand our choice is made such that each have the same probability of taking the value zero. In the setting where we have no information on the smoothness of the $X_i$ in the form of a bound on $u$, our geometric approximation bound improves upon this for sufficiently large $\mathbb{E}[X_1^2]$, and also applies in the case where $N$ is not geometric. We note also that both geometric approximation results have the desirable property that the upper bound is zero whenever $\mu=1$. Finally, we refer to \cite{p11} for related results on the exponential approximation of random sums in the case where the $X_i$ are no longer assumed to be integer-valued.

\subsection{Application: Compound binomial risk process}\label{subsec:ruin}

We now consider how our Theorem \ref{thm:sum} may be applied to the compound binomial risk process in discrete time \cite{g88}, defined as follows: Let $m\in\mathbb{Z}^+$ and let $\eta,\eta_1,\eta_2,\ldots$ be independent and identically distributed random variables supported on $\mathbb{Z}^+$. The risk process $\{U_t:t=0,1,\ldots\}$ is defined by 
\[
U_t=m+t-\sum_{i=1}^t\eta_i\,,
\] 
where we write $q=\mathbb{E}[\eta]$ and make the assumption that $q<1$ to ensure a positive drift of this process. We will also assume the existence of second and third moments of $\eta$ as required. Letting $\tau=\min\{t\geq1:U_t\leq0\}$, our interest is in the infinite-horizon ruin probability $\psi(m)$ defined by
\begin{equation}\label{ruin_prob_def}
\psi(m)=\mathbb{P}(\tau<\infty|U_0=m).
\end{equation} 
This ruin probability is given by a discrete Pollaczek--Khinchine formula: We have $\psi(0)=q$ and
\begin{equation}\label{gd1}
\psi(m)=\mathbb{P}\left(\sum_{i=1}^MY_i\geq m\right)
\end{equation}
for $m=1,2,\ldots$, where $M\sim\mbox{Geom}(1-q)$ and $Y,Y_1,Y_2,\ldots$ are independent and identically distributed with $\mathbb{P}(Y=j)=q^{-1}\mathbb{P}(\eta>j)$ for $j=0,1,\ldots$; see Proposition 2.6 of \cite{s20}. We note that \cite{s20} gives some numerical approximations for $\psi(m)$, but without any explicit error bounds. Our goal in this section is to give explicit approximations for $\psi(m)$, complete with error bounds, that are significantly easier to work with than the tail probabilities in \eqref{gd1}.

The geometric sum in \eqref{gd1} does not quite fit into the framework of Theorem \ref{thm:sum}, since here the summands are allowed to take the value zero with positive probability. We therefore rewrite this by letting $X,X_1,X_2\ldots$ be independent and identically distributed with 
\[
\mathbb{P}(X=j)=\mathbb{P}(Y=j|Y>0)=\frac{\mathbb{P}(\eta>j)}{q-\mathbb{P}(\eta>0)}
\]
for $j=1,2,\ldots$, and letting $N\sim\mbox{Geom}(r)$ with
\begin{equation}\label{rdef}
r=\frac{1-q}{\mathbb{P}(\eta=0)}\,.
\end{equation}
The following lemma is well known (it follows, for example, from a stronger observation on page 165 of \cite{m87}), but we include its short proof to keep our exposition self-contained.
\begin{lemma}\label{lem:geomsum}
Letting $W=\sum_{i=1}^NX_i$, with $N$ and the $X_i$ as above, we have $\mathcal{L}(W)=\mathcal{L}\left(\sum_{i=1}^MY_i\right)$.
\end{lemma}
\begin{proof}
Let $K_1,K_2,\ldots$ be independent and identically distributed, with $K_j\sim\mbox{Bin}(j,\mathbb{P}(Y>0))$ having a binomial distribution for each $j=1,2,\ldots$. For $m\geq0$ we write 
\[
\mathbb{P}\left(\sum_{i=1}^MY_i>m\right)=(1-q)\sum_{j=1}^\infty q^j\mathbb{P}(Y_1+\cdots+Y_j>m)\,.
\]
Of the $j$ summands within the latter probability, we have $K_j$ that are non-zero, with the remaining $j-K_j$ taking the value zero. Conditioning on these $K_j$ and noting that those $Y_i$ that are conditioned to be non-zero have the same distribution as $X$, we obtain
\begin{align*}
\mathbb{P}\left(\sum_{i=1}^MY_i>m\right)&=(1-q)\sum_{j=1}^\infty q^j\sum_{k=0}^j\binom{j}{k}\mathbb{P}(Y>0)^k\mathbb{P}(Y=0)^{j-k}\mathbb{P}(X_1+\cdots+X_k>m)\\
&=(1-q)\sum_{k=0}^\infty\left(\sum_{j=k}^\infty\binom{j}{k}[q\mathbb{P}(Y=0)]^j\right)\left(\frac{\mathbb{P}(Y>0)}{\mathbb{P}(Y=0)}\right)^k\mathbb{P}(X_1+\cdots+X_k>m)\\
&=\frac{(1-q)}{1-q\mathbb{P}(Y=0)}\sum_{k=0}^\infty\left(\frac{q\mathbb{P}(Y>0)}{1-q\mathbb{P}(Y=0)}\right)^k\mathbb{P}(X_1+\cdots+X_k>m)\\
&=r\sum_{k=1}^\infty(1-r)^k\mathbb{P}(X_1+\cdots+X_k>m)=\mathbb{P}(W>m)\,,
\end{align*} 
as required.
\end{proof}

In order to state a bound on $\psi(m)$ we will need some further properties of the random variable $X$, which we collect in the following lemma.
\begin{lemma}\label{lem:moments}
With $X$ as above, 
\begin{align*}
\mathbb{E}[X]&=\frac{\mathbb{E}[\eta(\eta-1)]}{2(q-\mathbb{P}(\eta>0))}\,,\qquad
\mathbb{E}[X^2]=\frac{\mathbb{E}[\eta(\eta-1)(2\eta-1)]}{6(q-\mathbb{P}(\eta>0))}\,,
\end{align*}
and
\begin{align*}
d_{TV}(\mathcal{L}(X),\mathcal{L}(X+1))&=\frac{\mathbb{P}(\eta>0)}{2(q-\mathbb{P}(\eta>0))}\,.
\end{align*}
\end{lemma}
\begin{proof}
For $l=1,2$ we write
\[
\mathbb{E}[X^l]=\frac{1}{q-\mathbb{P}(\eta>0)}\sum_{j=1}^\infty j^l\mathbb{P}(\eta>j)=\frac{1}{q-\mathbb{P}(\eta>0)}\sum_{i=1}^\infty\mathbb{P}(\eta=i)\sum_{j=1}^{i-1}j^l\,,
\]
from which the first two expressions in the lemma follow. For the total variation distance, the definition \eqref{dtvdef} gives
\begin{align*}
d_{TV}(\mathcal{L}(X),\mathcal{L}(X+1))&=\frac{1}{2(q-\mathbb{P}(\eta>0))}\sum_{j=1}^\infty|\mathbb{P}(\eta>j-1)-\mathbb{P}(\eta>j)|\\
&=\frac{1}{2(q-\mathbb{P}(\eta>0))}\sum_{j=1}^\infty\mathbb{P}(\eta=j)\,,
\end{align*}
as required.
\end{proof}
We are now in a position to state the main result of this section.
\begin{theorem}\label{thm:ruin}
With $\psi(m)$ the ruin probability defined in \eqref{ruin_prob_def} above,
\begin{equation}\label{ruin1}
\left(\frac{q-\mathbb{P}(\eta>0)}{\mathbb{P}(\eta=0)}\right)^m\leq\psi(m)\leq\left(\frac{q-\mathbb{P}(\eta>0)}{\mathbb{P}(\eta=0)}\right)^m+\frac{1}{\mathbb{P}(\eta=0)}\left(\frac{1}{2}\mathbb{E}[\eta(\eta-1)]-q+\mathbb{P}(\eta>0)\right)\,,
\end{equation}
and
\begin{equation}\label{ruin2}
\left|\psi(m)-\left(\frac{\mathbb{E}[\eta(\eta-1)]}{2(1-q)+\mathbb{E}[\eta(\eta-1)]}\right)^m\right|\leq\frac{1}{2}\min\left\{1,\frac{(1-q)(1+v)}{\mathbb{P}(\eta=0)}\right\}\left(\frac{\mathbb{E}[\eta(\eta-1)(2\eta-1)]}{3\mathbb{E}[\eta(\eta-1)]}-1\right)\,,
\end{equation}
for $m=1,2,\ldots$, where
\[
v=\sqrt{\frac{-4(q-\mathbb{P}(\eta>0))}{(2q-3\mathbb{P}(\eta>0))\log\left(\frac{q-\mathbb{P}(\eta>0)}{\mathbb{P}(\eta=0)}\right)}}\,.
\]
\end{theorem}
\begin{proof}
From Lemma \ref{lem:geomsum} and the Pollaczeck--Khinchine formula \eqref{gd1} we may write that $\psi(m)=\mathbb{P}(W\geq m)$ for $m=1,2,\ldots$. With this representation, we employ the bound \eqref{geomsum} of Theorem \ref{thm:sum} in conjunction with the formula for $\mathbb{E}[X]$ in Lemma \ref{lem:moments} to obtain that
\[
\left|\psi(m)-\left(\frac{q-\mathbb{P}(\eta>0)}{\mathbb{P}(\eta=0)}\right)^m\right|\leq\frac{1}{\mathbb{P}(\eta=0)}\left(\frac{1}{2}\mathbb{E}[\eta(\eta-1)]-q+\mathbb{P}(\eta>0)\right)\,.
\]
The proof of \eqref{ruin1} is completed by noting that $W$ is stochastically larger than our approximating geometric random variable, and so $\mathbb{P}(W\geq m)\geq\left(\frac{q-\mathbb{P}(\eta>0)}{\mathbb{P}(\eta=0)}\right)^m$. 

For \eqref{ruin2} we will use \eqref{geomsum2}. In line with the notation there, and using the total variation distance given in Lemma \ref{lem:moments}, we write $1-u=\frac{\mathbb{P}(\eta>0)}{2(q-\mathbb{P}(\eta>0))}$, so that
\[
u=\frac{2q-3\mathbb{P}(\eta>0)}{2(q-\mathbb{P}(\eta>0))}\,,
\]  
and so, with $r$ given by \eqref{rdef}, 
\[
v=\sqrt{\frac{-2}{u\log(1-r)}}\,.
\]
The result then follows from \eqref{geomsum2} and Lemma \ref{lem:moments}.
\end{proof}
Note that, since they are based on geometric approximation of a geometric sum, we would expect the bounds in Theorem \ref{thm:ruin} to perform well when the summands $X$ are typically not much larger than one, that is, when $\mathbb{P}(\eta>2)$ is small. To conclude this section, we illustrate this point and the bounds of Theorem \ref{thm:ruin} with three examples.
\begin{example}
Suppose that $\eta$ is supported on $\{0,1,2\}$, with $\mathbb{P}(\eta=0)=\alpha_0$, $\mathbb{P}(\eta=2)=\alpha_2$, and $\mathbb{P}(\eta=1)=1-\alpha_0-\alpha_2$. We will assume that $\alpha_2<\alpha_0$ to ensure that $q<1$. In this case, both the bounds of Theorem \ref{thm:ruin} identify that $\psi(m)=(\alpha_2/\alpha_0)^m$. We note, though, that \eqref{ruin1} applies for all $\alpha_2\in[0,\alpha_0)$, while the upper bound of \eqref{ruin2} requires $\alpha_2>0$ to be well-defined. In this example the expression for $\psi(m)$ also follows more directly from the observation that $(U_0,U_1,\ldots)$ is a random walk on the integers with initial value $U_0=m$, and independent increments supported on $\{-1,0,1\}$ and distributed as $1-\eta$.
\end{example}
\begin{example}
In the setting where $\eta\sim\mbox{Geom}(\alpha)$ has a geometric distribution, with $\alpha>1/2$ to ensure that $q<1$, it is known that 
\[
\psi(m)=\left(\frac{1-\alpha}{\alpha}\right)^{m+1}\,.
\]
See Section 2 of \cite{s20}. Neither of the estimates in Theorem \ref{thm:ruin} yield exactly this value; the approximations given by \eqref{ruin1} and \eqref{ruin2} are
\[
\psi(m)\approx\left(\frac{1-\alpha}{\alpha}\right)^{2m}\,,\;\text{ and }\;\psi(m)\approx\left(\frac{(1-\alpha)^2}{3\alpha^2-3\alpha+1}\right)^m\,,
\] 
respectively. The first of these gives the correct value for $m=1$, but the second is a better approximation for larger $m$.
\end{example}
\begin{example}
Suppose $\eta$ has a mixed Poisson distribution. That is, given some non-negative random variable $\xi$, suppose that $\eta|\xi\sim\mbox{Pois}(\xi)$ has a Poisson distribution. We will illustrate our bound \eqref{ruin1} in this case; the bound \eqref{ruin2} may be evaluated similarly, but we omit that to keep the discussion concise. We have that $q=\mathbb{E}[\xi]$, which we assume is less than 1. We also have $\mathbb{P}(\eta=0)=\mathbb{E}[e^{-\xi}]$, and $\mathbb{E}[\eta(\eta-1)]=\mathbb{E}[\xi^2]$. Hence, the approximation error in \eqref{ruin1}, i.e., the final term on the right-hand side, is
\begin{equation}\label{eg:pois}
\frac{1}{\mathbb{P}(\eta=0)}\left(\frac{1}{2}\mathbb{E}[\eta(\eta-1)]-q+\mathbb{P}(\eta>0)\right)=\frac{1}{\mathbb{E}[e^{-\xi}]}\left(\frac{1}{2}\mathbb{E}[\xi^2]-\mathbb{E}[\xi]+1-\mathbb{E}[e^{-\xi}]\right)\,.
\end{equation}
For example, if $\eta\sim\mbox{Pois}(\lambda)$ has a Poisson distribution for some small $\lambda>0$, writing $e^{\pm\lambda}\approx1\pm\lambda$, this error is approximately $\frac{1}{2}\lambda^2(1+\lambda)$. The approximation error \eqref{eg:pois} in this case is illustrated in Figure \ref{fig:2}.

\begin{figure}
\begin{center}
\includegraphics[scale=0.5]{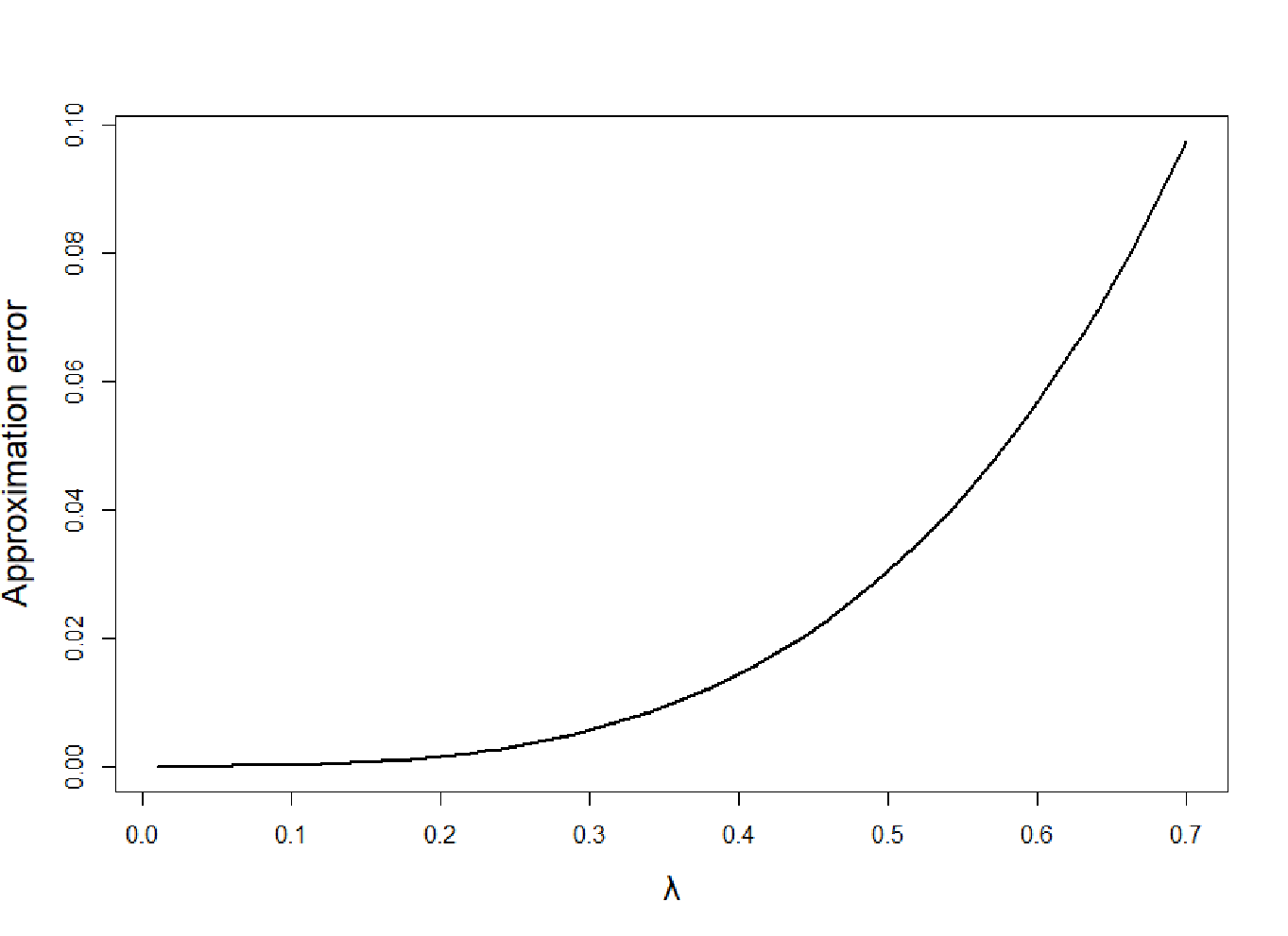}
\end{center}
\caption{The approximation error \eqref{eg:pois} as a function of $\lambda$}\label{fig:2}
\end{figure} 

As a second example, suppose instead that $\xi$ has a gamma distribution with density $x\mapsto\frac{\beta^\alpha}{\Gamma(\alpha)}x^{\alpha-1}e^{-\beta x}$ for $x>0$, where $\alpha>0$ and $\beta>0$ are parameters, so that $\eta$ has a negative binomial distribution. Then our approximation error is 
\begin{equation}\label{eg:gamma}
\left(1+\frac{1}{\beta}\right)^\alpha\left(\frac{\alpha(\alpha+1)}{2\beta^2}-\frac{\alpha}{\beta}+1\right)-1\,.
\end{equation} 
This converges to zero as $\alpha\to0$ for fixed $\beta$, or as $\beta\to\infty$ for fixed $\alpha$, for example. This can be seen on Figure \ref{fig:3}, which shows the approximation error \eqref{eg:gamma} as a function of $\beta$ for various values of $\alpha$.

\begin{figure}
\begin{center}
\includegraphics[scale=0.5]{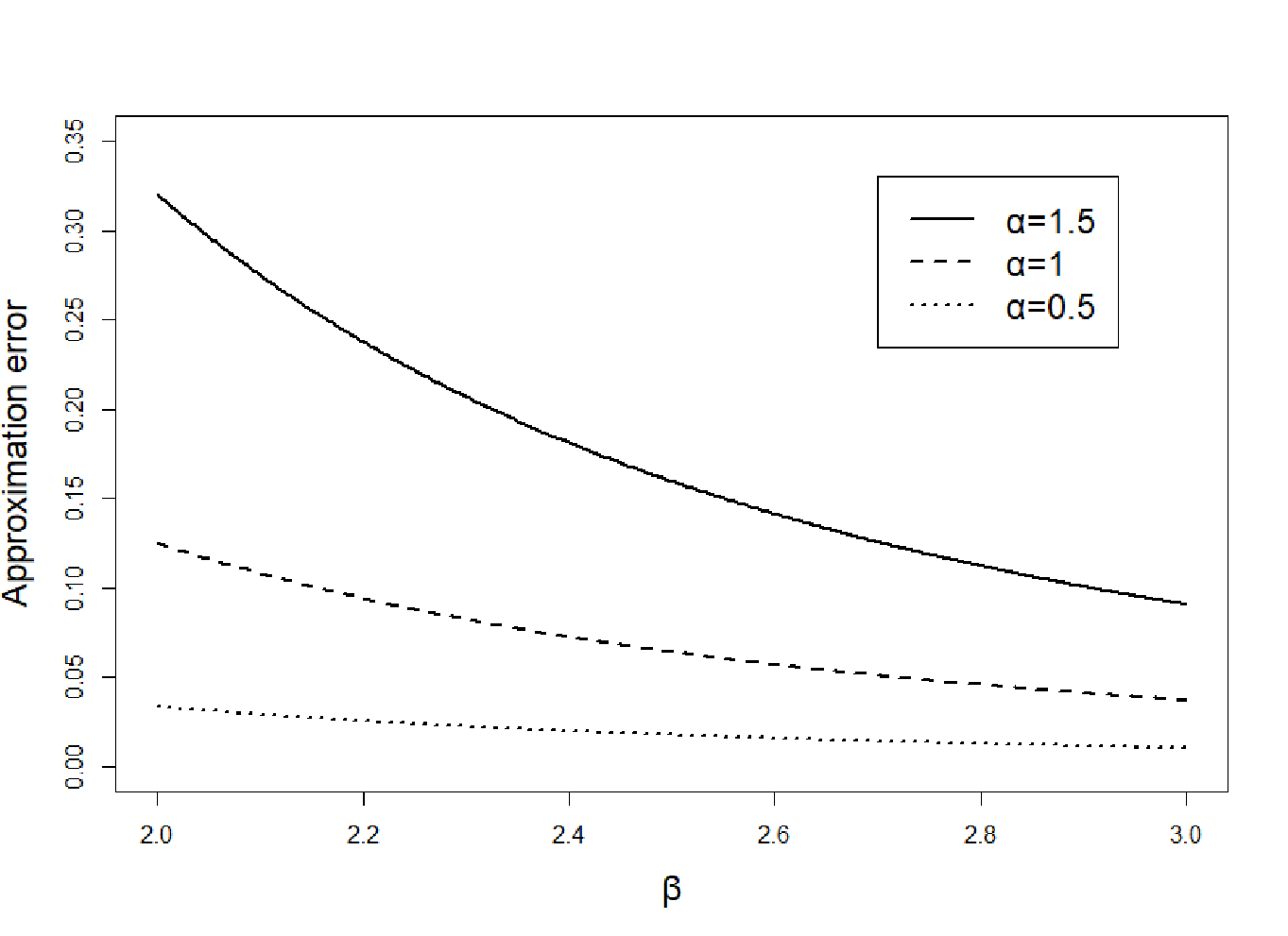}
\end{center}
\caption{The approximation error \eqref{eg:gamma} as a function of $\beta$ for $\alpha=0.5, 1, 1.5$}\label{fig:3}
\end{figure}
\end{example}


\vspace{12pt}
\noindent\textbf{Acknowledgements:}
F. Daly thanks the Belgian FNRS for financial support for a visit to the Universit\'e Libre de Bruxelles where this work started. C. Lef\`evre has received support from the {\it DIALog Research Chair} under the aegis of the Risk Foundation, an initiative of CNP. The authors thank an Editor and two anonymous reviewers for their careful readings of an earlier version of the paper and suggestions which improved the presentation of the work.

 
\bibliographystyle{harvard}

\end{document}